\documentclass{amsart}

\usepackage{xcolor} 
\usepackage{verbatim}
\usepackage{amsfonts}
\usepackage{amsmath}
\usepackage{amssymb}
\usepackage{amsthm}
\usepackage{enumerate}
\usepackage[utf8]{inputenc}
\usepackage{mathabx,epsfig}
\usepackage[bookmarks=false]{hyperref}
\usepackage{paralist, tabularx}
\usepackage[cmtip,arrow]{xy}
\usepackage{pb-diagram,pb-xy}

\newcommand{\TypeZero}{\mathrm{Type}(\emptyset)}
\newcommand{\TypeOm}{\mathrm{Type}(\omega)}
\newcommand{\TypemO}{\mathrm{Type}(\omega^*)}
\newcommand{\TypeOmO}{\mathrm{Type}(\omega +\omega^*)}

\newcommand{\N}{\mathbb{N}}
\newcommand{\M}{\mathbb{M}}
\newcommand{\LM}{\mathbb{L}}
\DeclareMathOperator{\Up}{Up}
\DeclareMathOperator{\Low}{Low}
\DeclareMathOperator{\Inc}{Inc}
\DeclareMathOperator{\im}{Im}
\DeclareMathOperator{\Lev}{Lev}

\DeclareMathOperator{\suc}{succ}
\DeclareMathOperator{\Fin}{Fin}

\newtheorem{theorem}{Theorem}
\numberwithin{theorem}{section}
\newtheorem{lemma}[theorem]{Lemma}

\newtheorem{fact}[theorem]{Fact}
\newtheorem{proposition}[theorem]{Proposition}

\newtheorem{question}[theorem]{Question}
\newtheorem{corollary}[theorem]{Corollary}
\newtheorem{definition}[theorem]{Definition}

\theoremstyle{definition}

\newtheorem{remark}[theorem]{Remark}

\title{On canonicity of almost linear minimal orders}

\author{Grzegorz Jagiella}

\address{Instytut Matematyczny Uniwersytetu Wrocławskiego, pl. Grunwaldzki 2, 50-384 Wrocław, Poland}
\address{ORCID: \href{http://orcid.org/0000-0002-5504-5260}{0000-0002-5504-5260}}

\email{grzegorz.jagiella@math.uni.wroc.pl}

\date{\today}
\keywords{minimal structures, definable orders, definable preorders, almost linear orders}

\subjclass[2020]{03C64 (Primary), 03C40, 06A06 (Secondary)}

\begin{document}

	\begin{abstract}
	We prove that if a minimal ordered structure $(M, <, \ldots)$ with infinite chains in an arbitrary language extending the language of strict orders interprets (in some power $M^n$) the linear order $(\omega, <)$, then this fact can be witnessed for $n=1$ via the incomparability relation of an $<$-definable strict order $R$; that $R$ is unique up to ``almost equality'', i.e., finite rearrangements of elements of $M$; and that $R$ can be defined from $<$ in a constructive way. We also show that several variants of the question whether a minimal ordered structure interprets an infinite linear order are all equivalent.
	\end{abstract}
	\maketitle
\section*{Introduction}

The motivation for this work is a question by Tanović \cite{Tan07} whether every minimal ordered structure with arbitrarily long chains interprets an infinite linear order. That question was in turn motivated by a conjecture of Pillay about the number of elementary extensions of an arbitrary countable first-order structure, up to isomorphism. Also, a specialized variant of the question for the so-called ``$\TypeOm$ structures'' appeared in \cite{KruTanWag13}, where the authors proved that the positive answer would fully resolve the remaining zero characteristic case of an old conjecture by Podewski \cite{Pod73}, which states that every minimal field is algebraically closed. Analysis of interpretability of linear orders in minimal structures was also used to partially resolve Kueker's Conjecture \cite{Tan12}.

In this paper, we do not settle the main question. We build upon some older work to show a number of results about a minimal ordered structure $M = (M, <, \ldots)$ (where $<$ is some distinguished partial order with infinite chains) which we assume interprets $(\omega, <)$. We show that the interpreting strict order is essentially unique up to so-called ``almost equality'', a notion made precise by Definition \ref{def:almost_equal}. We show that this order is moreover definable from $<$ alone. The proof gives an explicit iteration of definable relations and a finite-stabilization criterion for the existence of such an interpretation.

The original problem was stated somewhat ambiguously, as the meaning of ``interpretation'' was confined to dimension $1$: only considering the quotients of the first power of $M$ by some definable equivalence relation. Moreover, the work required for Podewski and Kueker conjectures only required a specialized variant of the question, restricting ordered structures to so-called $\TypeOm$ structures. We show that several natural variants of the question are all equivalent.

The paper is organized as follows. Section 1 contains general preliminaries. Section 2 establishes additional notation and recalls prior results by Tanović. Section 3 discusses interpretability in dimension 1 for $\TypeOm$ structures, the main results being Theorems \ref{thm:characterization} and \ref{thm:characterization2}. Section 4 is a short example of definability of orders in $(\omega, <)$. Section 5 proves a transfer result in the form of Proposition \ref{prop:interpretation_M_LM} and equivalences of natural variants of the main question in the form of Theorem \ref{thm:question_variants}.

\subsection*{AI statement} An earlier version of this manuscript and \cite{Jag14} were autoformalized to Lean 4 using Aristotle (\url{https://aristotle.harmonic.fun/}), which gave a suggestion to replace the ordering used in \cite[Proposition 13]{Jag14} with a colexicographical order. This resulted in a generalization stated as Proposition \ref{prop:n_to_one}. GPT-6 was used to correct minor mistakes in the final version of the manuscript.

\section{Preliminaries}

\subsection{Model theory}
We mostly use standard model-theoretic notations and conventions. Unless explicitly stated, we do not fix the language of the structures we consider. A definable set always means definable with parameters. Except for Section \ref{sec:n_to_one}, we set up the following conventions for the paper. First, we use a narrower definition of \emph{interpretation}: we say that a structure $M$ interprets another structure $N$ if there is a definable equivalence relation $E$ on $M$ such that the induced definable structure on $M/E$ is isomorphic to (some expansion of) $N$. This would be more properly called \emph{interpretation in dimension 1} (with domain of interpretation equal to all of $M$). Second, we usually do not make a typographic distinction between a structure and its universe.

\subsection{Notation and conventions for (pre)orders}

In the paper, we consider structures equipped with some number of orders and preorders. In this subsection we make clear some order-theoretic notations and conventions for such relations, most of which are textbook.

Except for Section \ref{sec:omega}, we always consider \underline{strict} partial orders: irreflexive, transitive, and asymmetric. In contrast, a \emph{preorder} always means a relation that is reflexive and transitive. Except in the preliminaries, we will keep these notions typographically distinct throughout the paper. We usually consider a structure $M$ equipped with a distinguished definable ``ambient'' strict order denoted $<$. Further definable strict orders on $M$ will be denoted by the symbols $R, S$, and preorders with $P, Q$, possibly with additional indicies or markings.

For a preorder $P$, define the $P$-\emph{bicomparability} relation $\sim_P$ by
\[x \sim_P y \iff P(x,y) \land P(y,x).\]
$P$-\emph{bicomparability} is an equivalence relation. Further define the \emph{strict part} of $P$, denoted $<_P$, by
\[x <_P y \iff P(x,y) \land \lnot P(y,x),\]
(equivalently $x <_P y \iff P(x,y) \land \lnot x \sim_P y$). Then $<_P$ is a strict partial order (induced by $P$).

For either a preorder or a strict order $R$, denote by $\parallel_R$ the $R$-incomparability relation, i.e.
\[x \parallel_R y \iff \lnot\left(R(x,y) \lor R(y,x)\right).\]
Note that for strict orders $x \parallel_R x$ for each $x$.

For $a \in M$, define its \emph{upper cone}, \emph{lower cone}, and \emph{incomparable set}:
\begin{align*}
\Up_R(a) & :=\{x \in M : R(a,x)\},\\
\Low_R(a) & :=\{x \in M : R(x,a)\},\\
\Inc_R(a) & :=\{x \in M : x \parallel_R a\}.
\end{align*}
All of these relations and sets are definable from their respective (pre)orders.

\subsection{Interpreting orders}

The central questions considered in this paper concern the interpretation of infinite linear orders in (minimal) ordered structures. Suppose that $M$ interprets a strict linear order $(A,<')$ via a definable equivalence relation $E$. Then the pullback $R$ of $<'$ by the quotient map is a definable strict partial order on $M$ such that $\parallel_{R} = E$. Then $A$ is infinite if and only if $R$ has an infinite chain.

Conversely, suppose that $R$ is a definable strict order on $M$ such that $\parallel_{R}$ is an equivalence relation. Then $R$ induces a strict linear order on $M/\parallel_{R}$ and the quotient map is an epimorphism of strict orders (again, $R$ has infinite chains if and only if $M/\parallel_{R}$ is infinite).

\begin{definition}
\label{def:order_interprets}
We say that a definable strict order $R$ on $M$ \emph{interprets} a linear order $(A, <)$ if $\parallel_{R}$ is an equivalence relation on $M$ and
\[(M/\parallel_{R}, R/\parallel_{R}) \approx (A, <).\]
\end{definition}
\begin{remark}
\label{rem:almost_linear_antichains}
\begin{enumerate}[(i)]
\item $R$ interprets a linear order if and only if for each $a \in M$ the set $\Inc_R(a)$ is an antichain.
\item $M$ admitting an order that interprets an infinite linear order is called \emph{almost linear} in \cite{KruTanWag13}.
\end{enumerate}
\end{remark}

\section{Minimal ordered structures}

Recall that a structure $M$ is \emph{minimal} if every definable subset of $M$ is either finite or cofinite. In particular, an infinite definable subset of $M$ is cofinite. We consider $M = (M, <, \ldots)$, an infinite minimal structure (in some language) with a distinguished definable strict order $<$.
When studying definable orders and preorders on $M$, we want to make some natural identifications: we will not distinguish between two orders if one can be obtained from another by some finite rearrangement of elements. More precisely:
\begin{definition}
\label{def:almost_equal}
We say that two binary relations $R$ and $S$ on $M$ are \emph{almost equal} if there is a cofinite subset $C \subseteq M$ such that $R$ and $S$ agree on $C$: that is, $R(a, b) \iff S(a, b)$ for all $a, b \in C$.
\end{definition}

Almost equality is an equivalence relation. Since $M$ is minimal, sets of the form $F(a, M)$ and $F(M, a)$ for a definable relation $F$ and $a \in M$ are either finite or cofinite. Hence whenever definable relations $R$ and $S$ are almost equal, their symmetric difference is definable in the language of equality. In particular, $R$ and $S$ are interdefinable.

\begin{remark}
By minimality, $M$ has the unique non-algebraic type $p$ over $M$. Suppose $\bar M \succ M$ is a monster model. Then $R$ and $S$ are almost equal if and only if they agree on $p(\bar M)^2$.
\end{remark}

We now recall a number of results by Tanović \cite{Tan05, Tan07} that explain the fundamental structure of $M$, adapted to our notation.

\begin{fact}
Let $(M, <)$ be minimal.
\begin{enumerate}[(i)]
\item $M$ has an infinite chain if and only if $M$ has finite chains of arbitrary length.
\item \cite[Proposition 1.1]{Tan05} If $M$ has an infinite chain, then $M$ is countable.
\end{enumerate}
\end{fact}

For each $a \in M$, the definable sets $\Up_<(a), \Low_<(a)$ and $\Inc_<(a)$ partition $M$, hence exactly one of them is cofinite. Now define the ``lower'', ``upper'' and ``incomparable'' parts of $M$:
\begin{align*}
L_<(M) & := \{a \in M : \Up_<(a) \text{ is cofinite}\},\\
U_<(M) & := \{a \in M : \Low_<(a) \text{ is cofinite}\},\\
I_<(M) & := \{a \in M : \Inc_<(a) \text{ is cofinite}\}.
\end{align*}
(As before, if the order is fixed, we will drop the subscript in the notation above.) $M$ is the disjoint union of the parts $L_<(M), U_<(M)$, and $I_<(M)$. These sets need not be definable. Tanović shows that if $I_<(M)$ is infinite, then it is cofinite and we have the following classification of minimal ordered structures by their so-called ``type'':
\begin{enumerate}
\item $\TypeZero$ structure: $M = I(M) \cup \text{finite set}$.
\item $\TypeOm$ structure: $M = L(M) \cup \text{finite set}$.
\item $\TypemO$ structure: $M = U(M) \cup \text{finite set}$.
\item $\TypeOmO$ structure: $L(M)$ and $U(M)$ both infinite, $I(M)$ finite.
\end{enumerate}
A $\TypeZero$ structure is an antichain (except for a finite set). The archetypical examples for $\TypeOm$, $\TypemO$, $\TypeOmO$ orders are respectively: $(\omega, <)$; $(\omega^*, <)$ ($\omega$ with reversed order); and $(\omega + \omega^*, <)$ ($\omega$ concatenated with $\omega^*$). In fact, these three examples are unique ``normalized'' (defined below) linear orders within their respective classes.

By replacing the original order with an almost equal one, we can move any finite number of elements between the three parts. For instance, we can take any $a \in M$ and define $<'$ by declaring $a$ as its smallest element and letting $<'$ agree with $<$ on $M \setminus \{a\}$. Then we have $a \in L_{<'}(M)$, while outside of $\{a\}$, the three parts $L_<(M), U_<(M)$, and $I_<(M)$ agree with $L_{<'}(M), U_{<'}(M)$, and $I_{<'}(M)$ respectively. In particular, we can adjust the original order to force each of $L_<(M), U_<(M), I_<(M)$ to be either infinite or empty. These adjustments do not change the type of the structure. Moreover, results about interpretability and definability of orders remain valid. When working with a given order, it is often convenient to replace it with one where each of the three parts is either infinite or empty. We will do this in a standardized way:
\begin{definition}
Let $<$ be a definable strict order on $M$. The \emph{normalization} of $<$ is a definable strict order $<'$ obtained in the following way:
\begin{compactenum}[(i)]
\item If $(M, <)$ is $\TypeOm$, rearrange $I_<(M) \cup U_<(M)$ into an antichain and prepend it to $L_<(M)$. Then $L_{<'}(M) = M$.
\item If $(M, <)$ is $\TypemO$, rearrange $I_<(M) \cup L_<(M)$ into an antichain and append it to $U_<(M)$. Then $U_{<'}(M) = M$.
\item If $(M, <)$ is $\TypeOmO$, rearrange $I_<(M)$ into an antichain and prepend it before $L_<(M) \cup U_<(M)$. Then $I_{<'}(M) = \emptyset, L_{<'}(M) = L_{<}(M) \cup I_{<}(M), U_{<'}(M) = U_{<}(M)$.
\item (pro forma) If $(M, <)$ is $\TypeZero$, let $<'$ be empty.
\end{compactenum}
We will say that an order $<$ is \emph{normalized} if it is its own normalization.
\end{definition}
An order is almost equal to its normalization. Normalized linear orders can be characterized by the types of infinite chains that appear within them. This characterization can also precisely distinguish the type of a minimal ordered structure (and justify the terminology).
By \cite[Theorem 2]{Tan07} and subsequent discussion:
\begin{fact}
\label{fact:chains_type}
For a minimal $(M, <)$:
\begin{enumerate}[(i)]
\item $M$ contains an infinite ascending chain if and only if $M$ is either $\TypeOm$ or $\TypeOmO$.
\item $M$ contains an infinite descending chain if and only if $M$ is either $\TypemO$ or $\TypeOmO$.
\end{enumerate}
Moreover, a normalized $\TypeOm$ structure does not contain a chain of (order-theoretic) order type $\omega + 1$, and a normalized $\TypemO$ structure does not contain a chain of order type $1 + \omega^*$.
\end{fact}
Note that a structure of $\TypemO$ is just a $\TypeOm$ structure with reversed order. Moreover, the type of a non-trivial minimal ordered structure does not depend on the choice of the non-trivial order, up to its direction:
\begin{fact}[{\cite[Corollary]{Tan07}}]
\label{fact:type_fixed}
Suppose that $R$ and $S$ are definable strict partial orders on $M$ with infinite chains.
\begin{enumerate}[(i)]
\item If $(M, R)$ is $\TypeOm$, then $(M, S)$ is either $\TypeOm$ or $\TypemO$.
\item If $(M, R)$ is $\TypeOmO$, then $(M, S)$ is $\TypeOmO$.
\end{enumerate}
\end{fact}

A minimal ordered structure need not be linear. A ``classic'' example is $M = \{(i, n) : n \in \omega \land i < n\}$ with $(i, n) < (i', n') \iff n < n'$. Then $\Inc_<\left((i, n)\right) = \{(j, n) : j < n\}$. In this case, $<$ interprets $\omega$. There exist examples of minimal structures $(M, <)$ where $\parallel_{<}$ is not an equivalence relation \cite{KruTanWag13}. However, every known example that is not $\TypeZero$ still admits a definable order that interprets (an infinite) linear order. This leads to the following main question:
\begin{question}
Let $(M,<)$ be a minimal order with infinite chains. Does $M$ interpret an infinite linear order?
\end{question}
Note that any structure interpreted in a minimal structure (in dimension $1$) is again minimal. The question can be refined by the following observation:
\begin{fact}
Up to interdefinability, there are exactly two infinite minimal linear orders, namely $(\omega, <)$ and $(\omega + \omega^*, <)$.
\end{fact}
This yields precisely Question 3 from \cite{Tan07}. We will mostly consider the case where $M$ is of $\TypeOm$ (and implicitly the $\TypemO$ case). As we will see in Section \ref{sec:n_to_one}, this case is equivalent to the original question, and even its more general variant (see Theorem \ref{thm:question_variants}).

We finish the section with some additional vocabulary:

\begin{definition}
\label{def:omega_interpreting}
For a definable strict order $R$ on $M$ that interprets $(\omega, <)$ (in the sense of Definition \ref{def:order_interprets}), we write $\Lev_R \colon M \to \omega$ for the quotient map and define $\Lev^{-1}(n)$ to be the \emph{$n$-th level of $M$} (with respect to $R$).
\end{definition}
The levels of $M$ are precisely the classes of $\parallel_{R}$ and (by minimality) they are all finite.

\section{Orders interpreting $(\omega, <)$}
\subsection{Uniqueness for $\TypeOm$ orders}

First we observe that the type of the structure limits the possibility of interpreted orders:
\begin{lemma}
\label{lem:limits}
Suppose that $(M, <)$ is a minimal ordered structure with infinite chains that interprets an infinite linear order. Then:
\begin{enumerate}[(i)]
\item $M$ interprets $(\omega, <)$ if and only if $(M, <)$ is either $\TypeOm$ or $\TypemO$.
\item $M$ interprets $(\omega + \omega^*, <)$ if and only if $(M, <)$ is $\TypeOmO$.
\end{enumerate}
\end{lemma}
\begin{proof}
Let $R$ be the definable strict order on $M$ that interprets a linear order $(A, <')$. Then $(A, <')$ abstractly embeds into $(M, R)$ and the existence of infinite increasing or decreasing chains in $(A, <')$ determines the type of $(M, R)$ (up to direction of order) by Fact \ref{fact:chains_type}. This type determines (up to direction) the type of $(M, <)$ by Fact \ref{fact:type_fixed}.
\end{proof}

Throughout the section we will assume that $(M, <)$ is a minimal structure of $\TypeOm$. Our first goal is to show that if $M$ admits a definable strict order $R$ that interprets $\omega$, then $R$ is unique up to almost equality.

\begin{lemma}
\label{lem:eq_antichain}
If $E$ is a definable equivalence relation on $M$ with infinitely many classes, then all classes are finite and almost all $E$-classes are antichains with respect to $<$.
\end{lemma}
\begin{proof}
Since $E$ has infinitely many classes, there can be no infinite (hence cofinite) class. Let $\mathrm{Max} := \{a \in M : a \text{ is maximal in } a/E\}$. $\mathrm{Max}$ is definable and infinite, and therefore cofinite. Hence in almost every class there are no $a, b$ with $a < b$.
\end{proof}

\begin{proposition}
\label{prop:omega_extends} 
Suppose that $R$ is a definable strict order on $M$ that interprets $\omega$. Then $R$ extends $<$ on a cofinite subset of $M$. If $<$ is normalized, then for almost all $a \in M$, we have $\Up_<(a) \subseteq \Up_R(a)$.
\end{proposition}
\begin{proof}
First suppose that $<$ is normalized. Let
\[D := \{a \in M : \Inc_R(a) \text{ is an antichain with respect to} <\}.\]
By Lemma \ref{lem:eq_antichain} $D$ is cofinite. For any $a \in D$ and $b \in M$, whenever $a < b$ and $ \lnot a R b$, we have $b R a$. Now suppose for a contradiction that $C := \{a \in M : \exists b \in M (a < b \land \lnot a R b)\}$ is infinite. $C$ is definable, hence cofinite. Construct an infinite sequence $a_0 < a_1 < \ldots \in C \cap D$ inductively as follows. Let $a_0 \in C \cap D$ such that $\Up_<(a_0) \subseteq C \cap D$. Having constructed $a_0, \ldots, a_n$, since $a_n \in C \cap D$, we can choose $a_{n+1}$ witnessing $a_n < a_{n+1} \land a_{n+1} R a_n$. Then $a_0 < a_{n+1}$, hence $a_{n+1} \in \Up_<(a_0) \subseteq C \cap D$ as required.

By construction, $a_0, a_1, \ldots$ is an infinite decreasing sequence with respect to $R$. Then $\Lev_R(a_0), \Lev_R(a_1), \ldots$ is an infinite decreasing sequence in $(\omega, <)$, a contradiction.
For arbitrary $<$, apply this argument to its normalization, which agrees with $<$ on a cofinite subset.
\end{proof}

We note that in Lemma \ref{lem:eq_antichain}, $<$ can be replaced by any definable strict order. In Proposition \ref{prop:omega_extends}, it can be replaced by any definable strict order of $\TypeOm$. As a corollary we obtain:

\begin{proposition}\label{prop:uniqueness} Assume that $R$ is a definable strict order on $M$ that interprets $\omega$. Then up to almost equality $R$ is the largest definable order on $M$ of $\TypeOm$. In particular, any two definable strict orders on $M$ that interpret $\omega$ are almost equal.
\end{proposition}
\begin{proof}
This follows directly from Proposition \ref{prop:omega_extends} applied to an arbitrary definable strict order of $\TypeOm$ in place of $<$.
\end{proof}

Proposition \ref{prop:uniqueness} in particular asserts the existence of the largest (up to almost equality) definable strict order of $\TypeOm$, assuming that $M$ interprets $\omega$. We will now establish the essential converse. For a definable strict order $R$, let
\[R^+(x, y) := \Up_R(y) \subsetneq \Up_R(x).\]
Then $R^+$ is a definable strict order extending $R$.
\begin{proposition}
\label{prop:maximal_interprets}
Suppose that $R$ is a normalized definable strict order on $M$ of $\TypeOm$ with $R = R^+$. Then $R$ interprets $\omega$.
\end{proposition}
\begin{proof}
Take any $x \in M$. As $R$ is normalized, $\Inc_R(x)$ is finite. Suppose $y$ is an $R$-maximal element of $\Inc_R(x)$. We show $\Up_R(y) = \Up_R(x)$. Suppose $z \in \Up_R(y)$. Then $z$ and $x$ are $R$-comparable. If $R(z, x)$ then $R(y, x)$, a contradiction, so $z \in \Up_R(x)$. Hence $\Up_R(y) \subseteq \Up_R(x)$. As $R$ agrees with $R^+$, we have $\lnot R^+(x, y)$, i.e. $\Up_R(y)$ is not a proper subset of $\Up_R(x)$. Now suppose that $y$ is not $R$-minimal in $\Inc_R(x)$. Then there is some $y' \in \Inc_R(x)$ with $R(y', y)$. Then $\Up_R(y') \supsetneq \Up_R(y) = \Up_R(x)$, so $R^+(y',x)$, a contradiction. Hence for all $x \in M$ we have that $\Inc_R(x)$ is an $R$-antichain, so the result follows from Remark \ref{rem:almost_linear_antichains}(i).
\end{proof}

As a corollary we obtain the following characterization:
\begin{theorem}
\label{thm:characterization} Let $R$ be a definable strict order of $\TypeOm$.
\begin{enumerate}[(1)]
\item If $R$ is normalized, then $R$ interprets $\omega$ if and only if $R = R^+$.
\item If $R$ interprets $\omega$, then its almost equality class is the greatest class among definable $\TypeOm$ orders. Conversely, if there is a maximal such class, then it has a representative that interprets $\omega$.
\end{enumerate}
\end{theorem}
\begin{proof}
Item (1) and the first part of item (2) follow from Propositions \ref{prop:uniqueness} and \ref{prop:maximal_interprets}. For the ``conversely'' part, let $R$ be a normalized element of a maximal class. Then $R$ is almost equal to $R^{+}$. Let $C$ be a cofinite set on which $R$ and $R^{+}$ agree and let $a \in C$ such that $\Up_R(a) \subseteq C$ (e.g. take $a \in \bigcap_{x \notin C} \Up_R(x)$). Let $S$ be the definable strict order that agrees with $R$ on $\Up_R(a)$ with the elements of $M \setminus \Up_R(a)$ forming an antichain prepended before $\Up_R(a)$. Then $S$ is almost equal to $R$, $S$ is normalized, and $S = S^{+}$. The conclusion follows from item (1).
\end{proof}

\subsection{Extending preorders} In this section we give a constructive criterion for the existence of a definable strict order interpreting $\omega$. We again assume that $M$ is a minimal structure with a distinguished definable strict order $<$ such that $(M, <)$ is of $\TypeOm$. By Proposition \ref{prop:omega_extends}, if $M$ admits an order $S$ that interprets $\omega$, then $S$ extends $<$ almost everywhere. We will inductively construct a sequence of definable preorders extending $<$ which stabilizes after finitely many steps precisely if $M$ interprets $\omega$, in which case a normalization of the strict part of the final preorder interprets $\omega$. The essential step of the construction is the operation $R \mapsto R^+$. However, the construction becomes clearer after changing the vocabulary to preorders instead of strict orders. For a definable strict order $R$ that interprets $(A, <)$, the relation $P(x,y) := R(x,y) \lor x \parallel_{R} y$ is a definable linear preorder such that $R(x,y) \iff x <_P y$ and $x \parallel_R y \iff x \sim_P y$. Conversely, a definable linear preorder $P$ induces the strict part $<_P$ whose incomparability relation is an equivalence relation. This leads to the following characterization:
\begin{fact} \label{fact:order_preorder} $M$ interprets $\omega$ if and only if $M$ admits a definable linear preorder with infinite (strict) chains.
\end{fact}

Suppose that $P$ is a definable preorder on $M$. Define
\[P^+(x, y) := \Up_{<_P}(y) \subseteq \Up_{<_P}(x).\]
It is easy to check that $P^+$ is also a definable preorder on $M$ that extends $P$.

We will also use the following characterization of $P^+$. Let
\[A_P(x, y) \iff x \parallel_{P} y \land \lnot \exists z ~(x \parallel_{P} z \land y <_P z).\]
(i.e., $A_P(x, y)$ says that $y$ is a $<_P$-maximal element of $\Inc_P(x)$).

\begin{lemma} \label{lem:preoder_by_maximal} With the notation as above:
\begin{enumerate}[(i)]
\item $P^+(x, y)$ if and only if $P(x, y) \lor A_P(x, y).$
\item $<_{P^+}$ extends $<_P$.
\end{enumerate}
\end{lemma}
\begin{proof}
(i) ($\Rightarrow$): assume $P^+(x, y)$ and suppose $\lnot P(x, y)$. If $P(y, x)$, then $y <_P x$, so $x \in \Up_{<_P}(y) \subseteq \Up_{<_P}(x)$, a contradiction. So $y \in \Inc_P(x)$. Suppose there is $z \in \Inc_P(x)$ with $y <_P z$. Then $z \in \Up_{<_P}(y) \setminus \Up_{<_P}(x)$, a contradiction. Hence $y$ is $<_P$-maximal in $\Inc_P(x)$.\\
($\Leftarrow$): if $P(x, y)$ then clearly $P^+(x,y)$. If $A_P(x, y)$, take any $z \in \Up_{<_P}(y)$. Since $y$ is $<_P$-maximal in $\Inc_P(x)$, $x$ and $z$ are $P$-comparable. If $P(z, x)$, then $P(y, x)$, contradicting $A_P(x, y)$. Hence $x <_P z$, and so $z \in \Up_{<_P}(x)$.\\
(ii) $A_P$ only adds edges between $P$-incomparable elements, so $x <_P y \Rightarrow x <_{P^+} y$.
\end{proof}

By a \emph{chain} with respect to a preorder $P$ we mean a chain with respect to $<_P$.

\begin{lemma} \label{lem:preorder_stab} Let $P$ be a definable preorder on $M$ with infinite chains. If $P = P^+$, then $P$ is linear.
\end{lemma}
\begin{proof}
Every bicomparability class of a definable preorder with infinite chains is finite: otherwise it would be cofinite and would meet an infinite strict chain in at least two points. By Lemma \ref{lem:preoder_by_maximal}(ii), this also applies to $P^+$.

First we show that $I_{<_P}(M) = \emptyset$. Assume for a contradiction that $I_{<_P}(M) \neq \emptyset$ and take a $<_P$-maximal $b$ in this finite set. Then $\Up_{<_P}(b)$ is finite and disjoint from $I_{<_P}(M)$. For each $x \in \Up_{<_P}(b)$, the set $\Inc_P(x) \subseteq \Inc_{<_P}(x)$ is finite, and $\Inc_P(b)$ is cofinite. Thus
\[C := \Inc_P(b) \setminus \bigcup\{\Inc_P(x) : x \in \Up_{<_P}(b)\}\]
is cofinite. For any $a \in C$, comparability and transitivity force $a <_P x$ for every $x \in \Up_{<_P}(b)$, so $P^+(a,b)$. Excluding the finitely many elements that are ${P^+}$-bicomparable with $b$, these comparisons are strict. Hence $b \in U_{<_{P^+}}(M)$, which contradicts $P=P^+$. Now $\Inc_P(a)$ is finite for every $a \in M$. If it is nonempty, then $A_P(a,b)$ holds for a $<_P$-maximal $b \in \Inc_P(a)$, again contradicting $P=P^+$ by Lemma \ref{lem:preoder_by_maximal}(i).
\end{proof}

We will need the following supporting lemma about functions on $\omega$:
\begin{lemma}
\label{lem:not_minimal}
Suppose $N = (\omega, <, f)$ is a first-order structure with $<$ interpreted as the usual order and $f \colon \omega \to \omega$ a function such that $f(n) \geq n$ for all $n \in \omega$. Then $N$ is minimal if and only if $f(n) - n$ is eventually constant.
\end{lemma}
\begin{proof}
Suppose $N$ is minimal but $f(n) - n$ is not eventually constant. For $m \in \omega$ let $A_m := \{n \in \omega : f(n) - n = m\}$. Each $A_m$ is definable in $N$, but not cofinite, hence finite. Therefore $f(n) - n$ diverges to infinity. Let $k \in \omega$. Then there is $n_0 \in \omega$ such that for all $n \geq n_0$ we have $f(n) \geq n + k$. Hence $\im(f) \cap [0, n_0 + k)$ has at most $n_0$ elements. Since $k$ is arbitrary, $\im(f)$ is infinite but not cofinite, contradicting minimality. The other direction is clear.
\end{proof}

Recall that if $M$ interprets (in dimension $1$) a structure $N$, then $N$ is minimal.

\begin{proposition}
\label{prop:bounded_chains}
Suppose that $R$ is a definable strict order on $M$ that interprets $\omega$ and $P$ is a definable preorder on $M$ with infinite chains. Then there exists $B \in \N$ dependent only on $M$ and $P$ such that for almost all $a \in M$, any $R$-chain in $\Inc_P(a)$ has at most $B$ elements.
\end{proposition}
\begin{proof}
Let $F:=I_{<_P}(M)$, a finite set. Let $Q$ be a definable preorder that agrees with $P$ on $M\setminus F$, with $F$ (if nonempty) forming one bicomparability class placed before $M \setminus F$. Then all sets $\Inc_Q(x)$ are finite, and for $x \notin F$ we have $\Inc_P(x) \subseteq \Inc_Q(x) \cup F$. Define $f \colon \omega \to \omega$ by setting $f(n)$ to be the largest $m \geq n$ such that $n = m$ or there exist $x, y\in M$ with $\Lev_R(x) = n$, $\Lev_R(y) = m$, and $y \in \Inc_Q(x)$. This maximum exists because the levels and the sets $\Inc_Q(x)$ are finite. The structure $(\omega, <, f)$ is interpreted in $M$ and hence minimal. By Lemma \ref{lem:not_minimal}, there are $k, n_0 \in \omega$ such that $f(n) = n + k$ for all $n \geq n_0$.

Put $D := \{y \in M : \Lev_R(y) < n_0\}$. Outside of the finite set $F \cup D \cup \bigcup_{y \in D} \Inc_Q(y)$, every $x$ and every $y \in \Inc_Q(x)$ have levels at least $n_0$. By symmetry of incomparability we also have $|\Lev_R(y) - \Lev_R(x)| \leq k$. Hence any $R$-chain in $\Inc_Q(x)$ has at most $2k + 1$ elements, and $B := 2k + 1 + |F|$ works for $\Inc_P(x)$. Any two definable strict orders interpreting $\omega$ are almost equal, so their sufficiently high levels agree up to a constant difference. Since every $\Inc_Q(x)$ is finite, the obtained value of $k$ does not depend on this difference. Hence $B$ does not depend on the choice of $R$.
\end{proof}

We now iterate the operation $P \mapsto P^+$ to build a chain of definable preorders. Suppose $P$ is a definable preorder. Set $P_0 := P$ and $P_{i+1} = (P_i)^+$ for all $i \in \N$.

\begin{theorem}
\label{thm:characterization2}
Suppose $P = P_0$ has infinite chains. The following are equivalent:
\begin{enumerate}[(i)]
\item $P_{i+1} = P_i$ for some $i \in \omega$.
\item $M$ interprets $\omega$.
\end{enumerate}
\end{theorem}
\begin{proof}
(i) $\Rightarrow$ (ii) follows from Lemma \ref{lem:preorder_stab} and Fact \ref{fact:order_preorder}. For (ii) $\Rightarrow$ (i), let $R$ be a definable strict order on $M$ that interprets $\omega$. Let $I_i:=I_{<_{P_i}}(M)$. Since $<_{P_{i+1}}$ extends $<_{P_{i}}$, we have $I_{i+1}\subseteq I_i$. The argument in the proof of Lemma \ref{lem:preorder_stab} shows that if $I_i$ is nonempty, then any $<_{P_i}$-maximal element of $I_i$ belongs to $U_{<_{P_{i+1}}}(M)$. Hence after at most $t = |I_0|$ steps we get $I_t=\emptyset$. So for all $x \in M$ and $i \geq t$, the sets $\Inc_{P_i}(x)$ are all finite.

By Fact \ref{fact:type_fixed}, either all the orders $<_{P_i}$ are of $\TypeOm$, or all of them are of $\TypemO$.\\
\emph{Case 1: } All orders $<_{P_i}$ are of $\TypeOm$. For any $S = P_i$ with $i \geq t$, by Proposition \ref{prop:omega_extends} there exists a cofinite $C \subseteq M$ on which $R$ extends $<_S$. After removing the finite set $\bigcup_{z \notin C} \Inc_S(z)$ from $C$, we can find a cofinite $C' \subseteq C$ such that $\Inc_S(x)\subseteq C$ for every $x \in C'$. If $\Inc_S(x)$ is nonempty, each of its $R$-maximal elements is $<_S$-maximal, and becomes $S^+$-comparable with $x$ by Lemma \ref{lem:preoder_by_maximal}(i). Hence every $R$-chain in $\Inc_{S^+}(x)$ is strictly shorter than the longest $R$-chain in $\Inc_{S}(x)$. By Proposition \ref{prop:bounded_chains}, there is $B \in \N$ bounding the lengths of $R$-chains in $\Inc_{P_t}(x)$ for almost all $x$. Hence we get $\Inc_{P_{t+B}}(x) = \emptyset$ for almost all $x$. All remaining pairs of incomparable elements belong to some finite $F \times F$, so subsequent iterations can only add finitely more pairs, and the sequence stabilizes.\\
\emph{Case 2}: All orders $<_{P_i}$ are of $\TypemO$. Repeat the argument for Case 1, replacing $R$ with its reverese (taking a cofinite subset $C \subseteq M$ such that $R$ extends the reverse of $<_S$).
\end{proof}

In particular, if $M$ interprets $\omega$, we can apply Theorem \ref{thm:characterization2} to the preorder $P_0 := \leq$, the reflexive closure of the ambient order $<$, and obtain a linear preorder $P_i$. Its strict part has $\TypeOm$ because it extends $<$, and its normalization recovers the (unique up to almost equality) strict order that interprets $\omega$.

\begin{corollary}
\label{cor:order_interprets_order}
If $(M, <, \ldots)$ interprets a linear order by a definable strict order $R$, then $R$ is $<$-definable.
\end{corollary}

We finish the section with a characterization of $R$-incomparability: a weak dual to Proposition \ref{prop:omega_extends}.
\begin{proposition}
Suppose $R$ is a definable strict order on $M$ that interprets $\omega$. Then up to almost equality $\parallel_{R}$ is the coarsest definable equivalence relation on $M$ with infinitely many classes.
\end{proposition}
\begin{proof}
Directly from Lemma \ref{lem:eq_antichain}: if a definable equivalence relation $E$ has infinitely many classes, then for almost all $a \in M$ we have $a/E \subseteq \Inc_R(a)$.
\end{proof}

\section{Example: orders on $(\omega,\leq)$}\label{sec:omega}

In this section we illustrate the characterizations from the previous section in the most basic structure of $\TypeOm$, namely $(\omega, <)$. We precisely describe all $\TypeOm$ orders definable in $(\omega, \leq)$ up to almost equality. We will work with the reflexive closure $\leq$ rather than $<$. Let $\suc \colon \omega \to \omega, \suc(n) := n+1$ be the successor function.
\begin{fact}
$(\omega, \leq, \suc, 0)$ has quantifier elimination.
\end{fact}
We note that $\suc$ is definable in $(\omega, \leq)$. For $H \subseteq \N$, define
\[x \leq_H y \iff x \leq y \land y - x \in H.\]
When $H$ is cofinite, $\leq_H$ is definable in $(\omega,\leq)$.
\begin{fact}
\label{fact:submonoid}
$\leq_H$ is a partial order if and only if $H \subseteq (\N, +)$ is a submonoid.
\end{fact}
\begin{proof}
Reflexivity is equivalent to $0 \in H$. For transitivity, $a \leq_H b \land b \leq_H c \Rightarrow a \leq_H c$ is equivalent to $b - a \in H \land c - b \in H \Rightarrow c - a \in H$.
\end{proof}

\begin{lemma}
Let $R$ be an order definable in $(\omega, \leq)$ with infinite increasing chains (so that its strict part is $\TypeOm$). Then there is a cofinite submonoid $H \subseteq (\N, +)$ such that $R$ is almost equal to $\leq_H$.
\end{lemma}
\begin{proof}
By quantifier elimination in $(\omega, \leq, \suc, 0)$, $R(x,y)$ eventually depends only on the difference $y-x$. There is no $d > 0$ such that $R(x+d, x)$ holds for cofinitely many $x \in \omega$: otherwise there would exist an infinite descending $R$-chain, contradicting that $R$ is $\TypeOm$. Let $H := \{d \in \N: R(x, x+d) \text{ for almost all } x\}$. Then $R$ agrees with $\leq_H$ on a cofinite set. By Fact \ref{fact:submonoid}, $H$ is a submonoid of $(\N,+)$. It is clearly cofinite.
\end{proof}

\section{Interpretability in $n$ dimensions}\label{sec:n_to_one}

In this section we show equivalences between natural generalizations as well as specializations of the main question. A large part of the necessary work and some of the equivalences were already presented in \cite{Jag14}. The results of this section will make the picture complete. The results of the previous sections also allow us to shorten some arguments from the previous work. As a consequence, this section will be largely self-contained except for the following result:
\begin{theorem}[{\cite[Theorem 6]{Jag14}}]
\label{thm:stable_embedding}
Let $\M = (M, <)$ be a minimal structure in a pure order language with $L(M)$ infinite, and consider its substructure $\LM = (L(M), <)$. Then for all $n \in \N$:
\begin{enumerate}
\item Any subset of $L(M)^n$ definable in $\LM$ is a trace of a subset of $M^n$ definable in $\M$.
\item For any subset $X \subseteq M^n$ definable in $\M$, its trace $X \cap L(M)^n$ is definable in $\LM$.
\end{enumerate}
In particular, $\LM$ is minimal and stably embedded in $\M$.
\end{theorem}
Clearly we have $L(L(M)) = L(M)$, i.e., $(L(M), <)$ (where $<$ is restricted to $L(M)$) is of $\TypeOm$. The next proposition does not require minimality and is a straightforward generalization of \cite[Proposition 13]{Jag14}:
\begin{proposition}
\label{prop:n_to_one}
Let $M$ be an arbitrary structure that interprets an infinite linear order in some dimension $n > 0$. Then $M$ interprets an infinite linear order in dimension $1$. Moreover, if the original interpreted order is $(\omega,<)$, the order interpreted in dimension $1$ can also be chosen to be $(\omega,<)$.
\end{proposition}
\begin{proof}
We proceed by induction on $n$. Suppose $M$ interprets an infinite linear order in dimension $n$ with $n > 1$. Assume without loss of generality that the domain of interpretation is the whole $M^n$. Let $<$ be a definable strict order on $M^n$ such that $E := \parallel_{<}$ is an equivalence relation with infinitely many classes and let $A = (M^n/E, </E)$ be the quotient linear order. For $m \in M$ define
\[C(m) := \{(m, m_2, \ldots, m_n)/E : m_2, \ldots, m_n \in M\}.\]
Then $C(m)$ is a nonempty chain in $A$. If for some $m \in M$ the set $C(m)$ is infinite, we define the relation $R$ on $M^{n-1}$ by $R(\bar m_1, \bar m_2) \iff m \bar m_1 < m \bar m_2$. Then $R$ is an order witnessing that $M$ interprets an infinite linear order in dimension $n-1$.

Otherwise, if for each $m \in M$ the set $C(m)$ is finite, let $<_{\text{colex}}$ be the colexicographical order on $\Fin(A)$, the family of finite subsets of $A$:
\[X <_{\text{colex}} Y \iff X \neq Y \text{ and } \forall x \in X \setminus Y ~ \exists y \in Y \setminus X ~x < y.\]
Then $<_{\text{colex}}$ is a linear order on $\Fin(A)$ and on $B := \{C(m) : m \in M\} \subseteq \Fin(A)$. $B$ is infinite since $\bigcup B = A$. Let $S$ be the pullback of $<_{\text{colex}}$ by the map $M \ni m \mapsto C(m) \in B$. Then $S$ is clearly a definable strict order on $M$ which interprets $B$.
For the ``moreover'' part, every infinite suborder of $\omega$ has order type $\omega$, and $(\Fin(\omega),<_{\text{colex}})$ has order type $\omega$ via $X\mapsto\sum_{k \in X}2^k$. Hence both branches of the induction preserve order type $\omega$.
\end{proof}
\begin{proposition}
\label{prop:interpretation_M_LM}
Let $\M = (M, <, \ldots)$ be a minimal relational structure with $<$ of $\TypeOmO$ and $\LM = (L(M), <, \ldots)$ its substructure. If $\LM$ is minimal and interprets an infinite linear order, then $\M$ interprets an infinite linear order.
\end{proposition}
\begin{proof}
Let $\M'$ and $\LM'$ be reducts of $\M$ and $\LM$ respectively to the pure language of orders. Assume that $\LM$ interprets an infinite linear order via a definable strict partial order $R$ on $L(M)$. Then $R$ is definable in $\LM'$ and by Theorem \ref{thm:stable_embedding}, there exists in $\M'$ an $<$-definable relation $S$ whose trace on $L(M)$ agrees with $R$. Let $\phi(x)$ say ``$S$ restricted to $\Low_<(x)$ is a strict order and $\parallel_{S}$ restricted to $\Low_<(x)$ is an equivalence relation''. Then $\phi(x)$ is a definable property that is true in $\M$ for all $x \in L(M)$. By minimality, there is some $a \in U(M)$ so that $\M \models \phi(a)$. Since $\Low_<(a)$ is cofinite, the conclusion follows.
\end{proof}

We now prove the equivalence between natural variants of the main question:
\begin{theorem}
\label{thm:question_variants}
The following are equivalent:
\begin{enumerate}[(1)]
\item Every minimal ordered structure with infinite chains interprets an infinite linear order in some dimension.
\item Every minimal ordered structure with infinite chains interprets an infinite linear order in dimension $1$.
\item Every minimal ordered structure of $\TypeOm$ interprets $(\omega, <)$ in some dimension.
\item Every minimal ordered structure of $\TypeOm$ interprets $(\omega, <)$ in dimension $1$.
\end{enumerate}
\end{theorem}
\begin{proof}
Proposition \ref{prop:n_to_one} establishes equivalence between (1) and (2) and between (3) and (4). Item (2) implies (4) using Lemma \ref{lem:limits}. We show (4) $\Rightarrow$ (1). Suppose $\M = (M, <, \ldots)$ is a minimal ordered structure with infinite chains. If $\M$ is not $\TypeOmO$, item (4) immediately applies (possibly after reversing the order). If $\M$ is of $\TypeOmO$, pass to the pure-order reducts $\M' = (M, <)$ and $\LM' = (L(M), <)$. Then $\M'$ is minimal and, by Theorem \ref{thm:stable_embedding}, so is $\LM'$. By assumption $\LM'$ interprets $(\omega,<)$, and Proposition \ref{prop:interpretation_M_LM} can be applied: $\M'$ (and therefore $\M$) interprets an infinite linear order. 
\end{proof}

\bibliographystyle{plain}
\bibliography{minimal2}

\end{document}